\documentclass[reqno,12pt]{amsart}
\usepackage{amsmath,amsthm,amsfonts,amssymb}
\usepackage{euscript}
\usepackage[utf8]{inputenc}
\usepackage[english]{babel}
\usepackage{graphicx}
\usepackage[colorlinks, citecolor=blue]{hyperref}
\date{}
\makeatletter
  \def\section{\@startsection{section}{2}%
    {\z@}{.5\linespacing\@plus.7\linespacing}{.5em}%
    {\normalfont\bfseries\centering}}
\def\@secnumfont{\bfseries}
\makeatother

\newcounter{z}

\newtheorem{theorem}{\indent Theorem}
\newtheorem{lemma}{\indent Lemma}
\newtheorem{proposition}{\indent Proposition}
\newtheorem{definition}{\indent Definition}

\begin{document}

\title{Classification of $(1{,}2)$-reflective anisotropic hyperbolic lattices of rank $4$}

\author{Nikolay V. Bogachev}

\address{Moscow Institute of Physics and Technology;}

\address{Lomonosov Moscow State University;}

\address{Adyghe State University.}


\email{nvbogach@mail.ru}

\begin{abstract}
A hyperbolic lattice is called \textit{$(1{,}2)$-reflective} if its automorphism group
is generated by $1$- and $2$-reflections up to finite index.
In this paper we prove that the fundamental polyhedron of a $\mathbb{Q}$-arithmetic
cocompact reflection group in the three-dimensional Lobachevsky space contains an
edge such that the distance between its framing faces is small enough.
Using this fact we obtain a
classification of $(1{,}2)$-reflective anisotropic hyperbolic lattices of rank $4$.
\end{abstract}

\maketitle



\tableofcontents

\section{Introduction}
\label{s1}

\subsection{Preliminaries}
\label{ss1.1}

By definition, a \textit{quadratic lattice} is a free Abelian group with a non-degenerate integral
symmetric bilinear form
called an \textit{inner product}.
A quadratic lattice $L$ is
called \textit{Euclidean}
if its inner product is positive definite and it is called \textit{hyperbolic}
if its inner product is a form of signature $(n, 1)$.
A lattice $L$ is said to be \textit{isotropic} if the corresponding quadratic form represents zero,
otherwise it is said to be \textit{anisotropic}.

A lattice $L$ is said to be \textit{even} if $(x,x) \in 2 \mathbb{Z}$ for all $x \in L$,
otherwise it is said to be \textit{odd}. Any odd lattice $L$ contains a unique even sublattice
of index $2$ formed by all its vectors whose lengths have even squares.

Let $L$ be a hyperbolic lattice. We shall assume that it is embedded into the
\textit{Minkowski space} $V=L \otimes \mathbb{R}=\mathbb{E}^{n, 1}$,
and we shall take one of the connected components of the hyperboloid
\begin{equation}
\label{eq1}
\{x \in \mathbb{E}^{n, 1}\mid (x, x)=-1\}
\end{equation}
as a \textit{vector model} of the $n$-\textit{dimensional (hyperbolic) Lobachevsky space} $\mathbb{L}^n$
with the metric $\rho$ given by the formula $\cosh\rho(v,w)= -(v,w)$.
In this case, the group $\mathrm{Isom}(\mathbb{L}^n)$ of
motions of the Lobachevsky space is a subgroup $O'(V)$ of index $2$ of the pseudo-orthogonal group $O(V)$ and
it consists of all
transformations leaving invariant each connected component of the hyperboloid \eqref{eq1}.
The \textit{planes} in the vector model of the Lobachevsky space are non-empty
intersections of the hyperboloid with subspaces of $V$. The \textit{points at infinity} in this
model correspond to isotropic one-dimensional subspaces of $V$.

A primitive vector $e$ of a quadratic lattice $L$ is
called a \textit{root} or, more precisely, a
$k$-\textit{root}, where $k=(e,e) \in \mathbb{N}$, if $2(e, x) \in
k\mathbb{Z}$ for all $x \in L$. (For $k \leqslant 2$ the last condition is fulfilled automatically.)
Every root $e$ defines an \textit{orthogonal reflection} (called a $k$-\textit{reflection} if $(e,e)=k$)
in the space
$L \otimes \mathbb{R}$
$$
\mathcal{R}_e\colon x \mapsto x-\frac{2(e, x)}{(e, e)} e,
$$
which preserves the lattice $L$. In the hyperbolic case, $\mathcal{R}_e$ defines a reflection with respect to
the hyperplane
$$
H_e=\{x \in \mathbb{L}^n \mid (x, e)=0\}
$$
in the space $\mathbb{L}^n$ called a \textit{mirror} of the reflection $\mathcal{R}_e$.

Suppose that $\mathcal{O}(L)$ is the group of automorphisms of a lattice $L$. It is known that
$$
\mathcal{O}'(L)=\mathcal{O}(L) \cap \mathcal{O}'(V)
$$
is a discrete group of motions of the Lobachevsky space, and
its fundamental polyhedron has finite volume (see Venkov's paper~\cite{Ven37}). We denote by $\mathcal{O}_r(L)$,
$\mathcal{O}^{(2)}_r(L)$, and $\mathcal{O}^{(1{,}2)}_r(L)$ the subgroups of
$\mathcal{O}'(L)$ generated by all
reflections, all  $2$-reflections, and all  $1$- and $2$-reflections in~$\mathcal{O}'(L)$, respectively.

The lattice $L$ is said to be \textit{reflective}, $2$-\textit{reflective}, or $(1{,}2)$-\textit{reflective}
if the subgroup $\mathcal{O}_r(L)$,
$\mathcal{O}^{(2)}_r(L)$, or $\mathcal{O}^{(1{,}2)}_r(L)$, respectively,
has finite index in~$\mathcal{O}'(L)$. The lattice $L$ is reflective, $2$-reflective, or $(1{,}2)$-reflective
if and only if the fundamental polyhedron of the group $\mathcal{O}_r(L)$,
$\mathcal{O}^{(2)}_r(L)$, or $\mathcal{O}^{(1{,}2)}_r(L)$, respectively, has finite volume in the Lobachevsky space
$\mathbb{L}^n$.

It is obvious that every finite extension of a $2$-reflective
(or $(1{,}2)$-reflective) hyperbolic lattice is also a
$2$-reflective (respectively, $(1{,}2)$-reflective) lattice. We also
note  that every $2$-reflective hyperbolic lattice
is $(1{,}2)$-reflective.

Discrete groups generated by reflections were determined by H.\,S.\,M.~Coxeter. He classified them on spheres
$\mathbb{S}^n$ and in Euclidean spaces~$\mathbb{E}^n$ in~1934 (see \cite{Cox34}). A systematic
study of reflection groups in Lobachevsky spaces was initiated by Vinberg
in~1967 (see \cite{Vin67}). In the same paper, an  arithmeticity criterion was given for discrete groups
generated by reflections, new methods  of studying hyperbolic
reflection groups were suggested and various examples of such groups were  constructed.
In~1972, Vinberg proposed an algorithm
(see \cite{Vin72}, \cite{Vin73}) that, given a lattice $L$, enables
one to construct the fundamental polyhedron of the group $\mathcal{O}_r (L)$ and determine thereby
the reflectivity of the lattice $L$. It is known that cocompact
(with a compact or, equivalently, bounded fundamental polyhedron)
discrete reflection groups (as well as arithmetic reflection groups) do not exist in the Lobachevsky spaces of
dimension
$\geqslant 30$ (Vinberg, 1984, see \cite{Vin84}).
It was also proved that there are no reflective hyperbolic lattices of rank
$n+1 > 22$ (F.~Esselmann, 1996, see\ \cite{Ess96}). In~2007, V.\,V.~Nikulin
(see\ \cite{Nik07}) finally proved that, up to conjugacy, there are
only finitely many maximal arithmetic reflection groups (see also \cite{ABSW2008}).
These results give hope that all maximal arithmetic hyperbolic
reflection groups and, in particular, reflective hyperbolic lattices can be classified.

V.\,V.~Nikulin (1979, 1981 and 1984, see \cite{Nik79,Nik81b,Nik84}) classified all
$2$-reflective hyperbolic lattices of
rank not equal to $4$, and after that in~2000\ (see~\cite{Nik00}) he found all reflective hyperbolic lattices of rank $3$
with square free discriminants. In~1998, (see \cite{Vin98} and \cite{Vin07}) E.\,B.~Vinberg
classified all $2$-reflective hyperbolic lattices of rank $4$.
Subsequently, D.~Allcock (2011, see \cite{All12}) classified at all
all reflective hyperbolic lattices of rank $3$.

In 1989--1993 in the papers~\cite{Sch89,SW92,Wal93}  R.~Scharlau and C.~Walhorn
presented a list of all maximal groups of the form $\mathcal{O}_r(L)$, where $L$
is a reflective isotropic hyperbolic lattice of rank $4$ or $5$.
A similar result was obtained in 2017 in the dissertation of I.~Turkalj (see  \cite{Tur17})
for lattices of rank $6$.

Finally, the author of this paper announced a classification of all maximal
$(1{,}2)$-reflective anisotropic hyperbolic lattices of rank $4$ (2016 and
2017, see \cite{Bog16} and \cite{Bog17}). A more complete history of the problem
 can be read in the recent survey of M.~Belolipetsky~\cite{Bel16}.

\subsection{Notation}
\label{ss1.2}
We introduce some notation:

1)~$[C]$ is a quadratic lattice whose inner product in some
basis is given by a symmetric matrix $C$;

2)~$d(L) := \det C$ is the discriminant of the lattice $L = [C]$;

3)~$L \oplus M$ is the orthogonal sum of the lattices $L$ and $M$;

4)~$[k]L$ is the quadratic lattice obtained from $L$ by multiplying all inner
products by $k \in \mathbb{Z}$;

5)~$L^*=\{x \in L \otimes \mathbb{Q} \mid \forall\, y \in L \ \ (x, y) \in
\mathbb{Z}\}$~ is the \textit{adjoint} lattice.

Let $P$ be a compact acute-angled polyhedron in $\mathbb{L}^3$ and
let $E$ be some edge of it.
We denote by $F_1$ and $F_2$ the faces of the polyhedron $P$
containing the edge $E$. Let
 $u_3$ and $u_4$ be the unit outer
normals to the faces $F_3$ and $F_4$ containing the vertices of the edge $E$,
but not the edge itself.

\begin{definition}
\label{d1.1}
The faces $F_3$ and $F_4$ are called the framing edges of the edge $ E $,
and the number $|(u_3, u_4)|$ is its width.
\end{definition}

We associate with the edge $E$ the
set $\overline \alpha = (\alpha_{12}, \alpha_{13}, \alpha_{23}, \alpha_{14}, \alpha_{24})$,
where $\alpha_{ij}$ is the angle between the faces $F_i$ and $F_j$.

\subsection{Main results}
\label{ss1.3}
The main results of this paper are the following two assertions, the second
of which is proved with the help of the first one.

\begin{theorem}
\label{t1.1}
The fundamental polyhedron of every $\mathbb{Q}$-arithmetic cocompact
 reflection group in $\mathbb{L}^3$ has an edge of width less than $4.14$.
\end{theorem}

Recall that a $\mathbb{Q}$-arithmetic reflection group is any finite index subgroup of a group
of the form $\mathcal{O}' (L)$. It can be cocompact only in the case where the lattice $L$ is anisotropic.

In fact, a stronger result is obtained. Namely, it is proved that
there is an edge of width
$ t_{\overline \alpha}$, where
$t_{\overline \alpha} \leq 4.14$ is a number depending on the set $\overline \alpha$ of
dihedral angles around this edge (see Theorem~\ref{t2.3}).

\begin{theorem}
\label{t1.2}
Any $(1{,}2)$-reflective anisotropic hyperbolic lattice of rank $4$ over $\mathbb{Z}$ is either isomorphic
to
$[-7] \oplus [1] \oplus [1] \oplus [1]$ or
$[-15] \oplus [1] \oplus [1] \oplus [1]$,
or to an even index $2$ sublattice of one of them.
\end{theorem}

Actually, these lattices are $2$-reflective (see \cite{Vin07}).

The author hopes that the \emph{method of the outermost edge} (see \S\,\ref{s2}) employed 
in this paper
will be applicable for classifying all reflective anisotropic hyperbolic lattices of rank $4$.

\medskip

The author regards  as his pleasant duty to express his deep gratitude to E.\,B.~Vinberg
for posing the problem and also for some
ideas, valuable advice, help, and attention.

\section{The method of the outermost edge and proof of Theorem~\ref{t1.1}}
\label{s2}

\subsection{Nikulin's method}
\label{ss2.1}
In this paper we shall use \textit{the method of the outermost edge},
which is a modification of \emph{the method of narrow parts of polyhedra}, applied by
V.\,V. Nikulin in his papers \cite{Nik80} and \cite{Nik00}.

\begin{definition}
\label{d2.1}
A convex polyhedron in the space $\mathbb{L}^n$ is the intersection of finitely many halfspaces
such that it has non-empty interior.
A generalized convex polyhedron is the intersection of  a family of halfspaces (possibly infinite)
such that any ball intersects only finitely many their boundary hyperplanes.
\end{definition}

\begin{definition}
\label{d2.2}
A generalized convex polyhedron is called acute-angled if all its dihedral angles do not exceed $\pi/2$.
A generalized convex polyhedron is called a Coxeter polyhedron if all its dihedral angles have the form
$\pi/k$, where $k \in \{2, 3, \dots, \infty\}$.
\end{definition}

It is known that fundamental domains of discrete reflection groups are generalized \textit{Coxeter polyhedra}
(see the papers \cite{Cox34},
\cite{Vin85}, and \cite{VS88}).

Here and throughout  by faces of a polyhedron we mean its $(n-1)$-dimensional faces.
The Gram matrix of a system of vectors $v_1, \dots, v_k$ will be denoted by $G(v_1, \dots, v_k)$.

In his earlier works (see Lemma~3.2.1 in~\cite{Nik80} and the proof
of Theorem~4.1.1 in~\cite{Nik81b}) V.\,V.~Nikulin proved the following
assertion\footnote{We present this assertion in a form convenient  for us, although
it was not formulated in this way anywhere.}.

\begin{theorem}
\label{t2.1}
Let $P$ be an
acute-angled convex polyhedron of finite volume in~$\mathbb{L}^n$.
Then there exists a face $F$ such that
$$
\cosh \rho (F_1, F_2) \le 7,
$$
for any faces $F_1$ and $F_2$ of $P$ adjacent to $F$, where
$\rho{(\,\cdot\,, \cdot\,)}$ is the metric in the Lobachevsky space\footnote{In Nikulin's papers, the squares of the lengths of the normals to faces are $2$, therefore, in his works
there is a bound $(\delta, \delta') \leq 14$.}.
\end{theorem}

In the proof of this assertion, the face $F$ was chosen as the outermost face
from some fixed point $O$ inside the polyhedron $P$. Notice that
this theorem enables us to bound at once the absolute value of the inner product of
outer normals to faces adjacent to the face $F$. Indeed, if $F_1$ and
$F_2$ intersect or are parallel, this value is equal to the cosine of
the dihedral angle between these faces, and if these faces diverge, then it
equals the hyperbolic cosine of the distance between them.

\subsection{The method of the outermost edge}
\label{ss2.2}
We have the following corollary of Theorem~\ref{t2.1}.

\begin{proposition}
\label{p2.1}
Each compact (i.\,e., bounded) acute-angled polyhedron
$P \subset \mathbb{L}^3$ contains an edge of width not greater than $7$.
\end{proposition}

\begin{proof}
Following V.\,V.~Nikulin (see\ \cite{Nik81b}), we consider an interior point $O$ in $P$. Let $E$ be
\textit{the outermost\footnote{In an acute-angled polyhedron,
the distance from the interior point to the face (of any dimension) is equal to the
distance to the plane of this face.}
edge from it}, and let $F$ be a face containing this edge.
Let $E_1$, $E_2$ be disjoint\footnote{Note that we consider the case where the framing
faces are divergent, since otherwise the absolute value of the inner product does not exceed one.}
edges of this face coming out from different vertices of $E$.

Let  $O'$ be the projection of  $O$ onto the face $F$. Note that
$O'$ is an interior point of this face, since otherwise the point
$O$ would lie outside of some dihedral angle adjacent to  $F$ (because the polyhedron $P$ is acute-angled).
Further, since $E$ is
the outermost edge of the polyhedron for $O$, then
$$
\rho(O, E) \ge \rho(O, E_i), \qquad i=1, 2.
$$
It follows from this and  the three
perpendiculars theorem that the distance between the point~$O'$
and the edge $E$ is not
less than the distance between this point and any other edge of the face $F$.
This means that the edge $E$ is the outermost edge to the point $O'$
inside the polygon $F$ and we can use Theorem~\ref{t2.1}.


Further, let $F_3$ and $F_4$ be the faces (with unit outer normals $u_3$ and $u_4$,
respectively) of the polyhedron $P$ framing the
outermost edge $E$ and containing the edges $E_1$ and $E_2$, respectively.
Clearly, the distance between the faces is not greater than the distance
between their edges.
Therefore,
$$
-(u_3,u_4)=\cosh \rho(F_3, F_4) \le \cosh \rho(E_1, E_2) \le 7.
$$
The proposition is proved.
\end{proof}

Let now $P$ be the fundamental polyhedron of the group
$\mathcal{O}_r(L)$ for an anisotropic hyperbolic lattice $L$
of rank $4$. The lattice $L$ is reflective if and only if the polyhedron
$P$ is compact (i.\,e., bounded).

Let $E$~be an edge (of the polyhedron $P$) of width not greater than $t$. By
Propositon~\ref{p2.1} we can ensure that $t \leq 7$ (if we take the outermost
edge from some fixed point $O$ inside the polyhedron $P$). Let $u_1$, $u_2$
be the roots of the lattice $L$ that are orthogonal to the faces containing
the edge $E$ and are the outer normals of these faces.
Similarly, let $u_3$, $u_4$ be the roots corresponding to the framing
faces. We denote these faces
by $F_1$, $F_2$, $F_3$, and $F_4$, respectively. If $(u_3, u_3) = k$, $(u_4, u_4) = l$, then
\begin{equation}
\label{eq2}
|(u_3, u_4)| \leq t\sqrt{kl} \leq 7 \sqrt{kl}.
\end{equation}

Since we solve the classification problem for $(1{,}2)$-reflective lattices,
we have to consider the fundamental polyhedra of arithmetic
groups generated by $1$- and $2$-reflections. In this case we are given bounds
on all elements of the matrix $G(u_1, u_2, u_3, u_4)$,
because all the faces $F_i$ are pairwise intersecting, excepting,
possibly, the pair of faces $F_3$ and $F_4$. But if they do not intersect,
then the distance between
these faces is bounded by inequality~\eqref{eq2}.
Thus, all entries of the matrix $G(u_1, u_2, u_3, u_4)$ are integer and bounded, so
there are only finitely many possible matrices
$G(u_1, u_2, u_3, u_4)$.

The vectors $u_1, u_2, u_3, u_4$ generate some sublattice $L '$
of finite index in the lattice $L$. More precisely, the lattice $ L $ lies between
the lattices $L'$ and $(L')^*$, and
$$
[(L')^* : L']^2=|d(L')|.
$$
Hence it follows that $|d(L')|$ is divisible by $[L:L']^2$. Using this, in each case we shall
find  for a lattice
$L'$ all its possible extensions of finite index.

To reduce the enumeration of matrices $G(u_1, u_2, u_3, u_4)$ we shall use some
additional considerations enabling us to get sharper
bounds on the number $|(u_3, u_4)|$ than in inequality~\eqref{eq2}.

\subsection{Bounds for the length of the edge
$E$ for a compact acute-angled polyhedron in $\mathbb{L}^3$}
\label{ss2.3}
In this subsection $P$ denotes a compact acute-angled polyhedron in the three-dimensional
Lobachevsky space $\mathbb{L}^3$.

Keeping the assumptions and notation of the previous sections, we denote the vertices of the edge $E$ by
$V_1$ and $V_2$. The dihedral angles between the faces $F_i$ and $F_j$ will be denoted by
$\alpha_{ij}$.

Let $E_1$ and $E_3$ be the edges of the polyhedron $P$ outgoing
from the vertex $V_1$ and let $E_2$ and $E_4$ be the edges outgoing from $V_2$ such that the edges
$E_1$ and $E_2$ lie in the face $F_1$. The length of the edge $E$ is
denoted by $a$, and the plane angles between the edges
$E_j$ and $ E $ are denoted by $\alpha_j$ (see Figure~\ref{fig1}).

\begin{figure}[!htp]
\begin{center}
\includegraphics{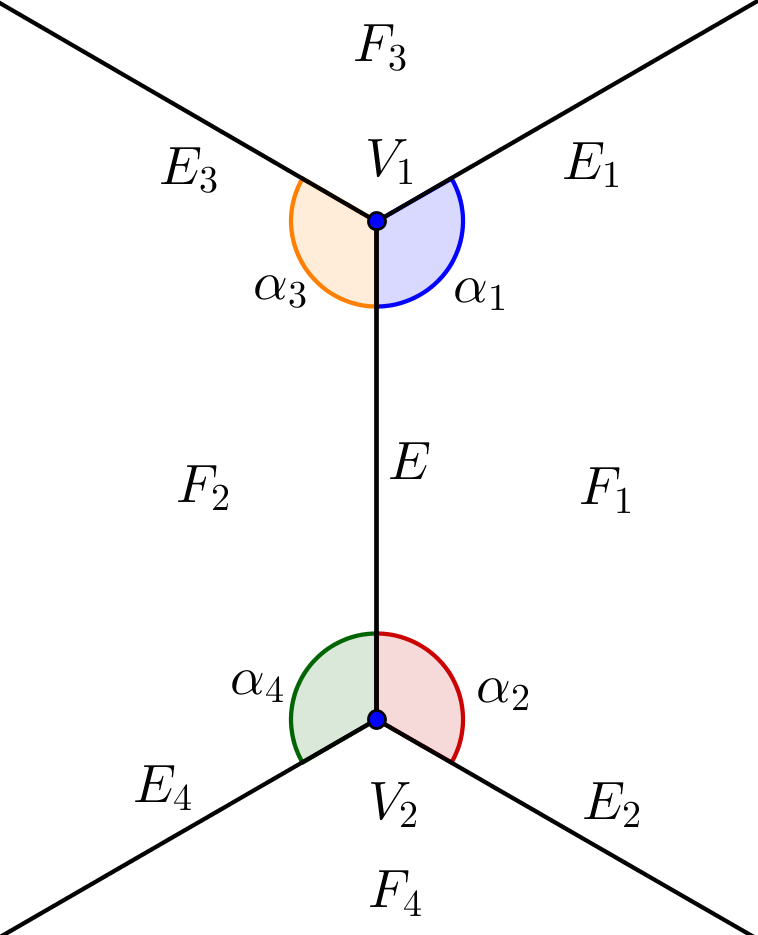}
\caption{The outermost edge}
\label{fig1}
\end{center}
\end{figure}

\begin{theorem}
\label{t2.2}
The length of the outermost edge satisfies the inequality
$$
a < \operatorname{arcsinh} \biggl(\frac{\operatorname{tanh} (\ln (\cot ({\alpha_{12}}/{4})))}
{\tan ({\alpha_3}/{2})}\biggr)+\operatorname{arcsinh} \biggl(\frac{\operatorname{tanh} (\ln (\cot
({\alpha_{12}}/{4})))}{\tan ({\alpha_4}/{2})} \biggr).
$$
\end{theorem}

\begin{proof}
Denote by $O_1$ and $O_2$ the orthogonal projections of the point $O$ onto the faces $F_1$ and
$F_2$, respectively. By the theorem of three perpendiculars, both points fall
under the projection onto this edge $E$ on the same point $A$, which is the projection
of the point $O$ onto this edge. Due to the fact that the polyhedron $P$ is acute-angled, the point
$A$ is an inner point of the edge~$E$.

Thus, we get a flat quadrilateral $AO_1 O O_2$, in which
$\angle A = \alpha_{12}$ (the dihedral angle between the faces $F_1$ and $F_2$),
$\angle O_1=\angle O_2=\pi/2$, $AO_1=a_1$, $AO_2=a_2$ (see\ Figure~\ref{fig2}).

\begin{figure}[!htp]
\begin{center}
\includegraphics[scale=1.4]{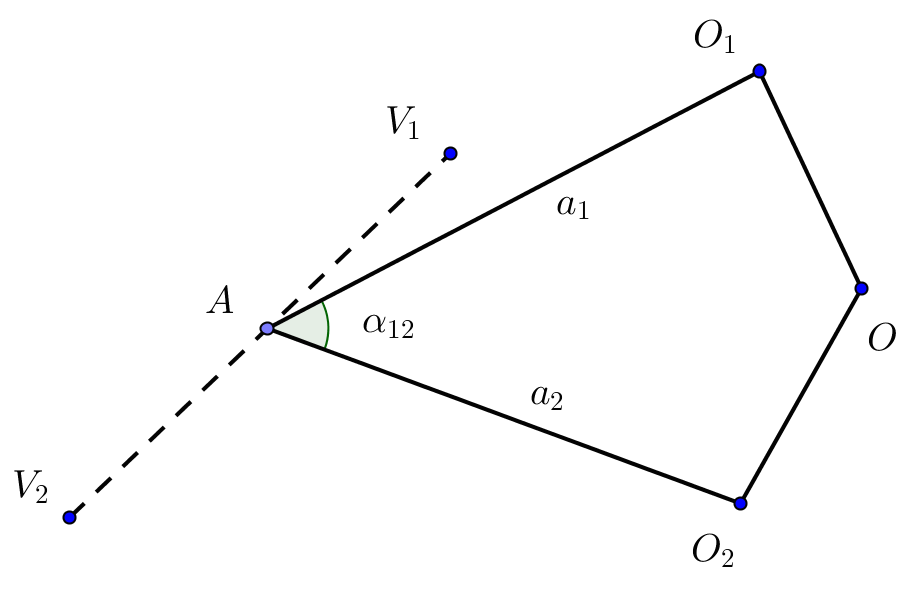}
\caption{A quadrilateral $AO_1 O O_2$}
\label{fig2}
\end{center}
\end{figure}

In the limiting case where the point $O$ is a point at infinity, the dihedral
angle~$\alpha_{12}$ is composed of the so-called \textit{angles of parallelism}
$\Pi(a_1)$ and~$\Pi(a_2)$. In our case $O \in
\mathbb{L}^3$, therefore,
$$\alpha_{12} < \Pi(a_1)+\Pi(a_2)=2 \arctan(e^{-a_1}) + 2 \arctan(e^{-a_2}).$$

Denote by $V_1 I$, $V_2 I$, $V_1 J$, $V_2 J$ the bisector of angles $\alpha_1$,
$\alpha_2$, $\alpha_3$, $\alpha_4$, respectively. Let $h_I$ and $h_J$ be
the distances from the points $I$ and $J$ to the edge $E$. Without loss of generality
 we can assume that
$h_J \leqslant h_I$.

Since the edge $E$ is the outermost edge for the point $O$, we have
\begin{equation}
\label{eq3}
\rho(O_1, E) \leq \rho(O_1, E_1), \rho(O_1, E_2), \qquad
\rho(O_2, E) \leq \rho(O_2, E_3), \rho(O_2, E_4).
\end{equation}
Then it is clear that $h_J \leq h_I \leq a_1$, $h_J \leq a_2$, since
inequalities~\eqref{eq3} imply that the points $O_1$ and $O_2$ lie inside flat
angles vertical to the angles $V_1 I V_2$ and $V_1 J V_2$, respectively (the scan of
faces around the edge $E$ is represented in Figure~\ref{fig3}).

\begin{figure}[!htp]
\begin{center}
\includegraphics[scale=1.9]{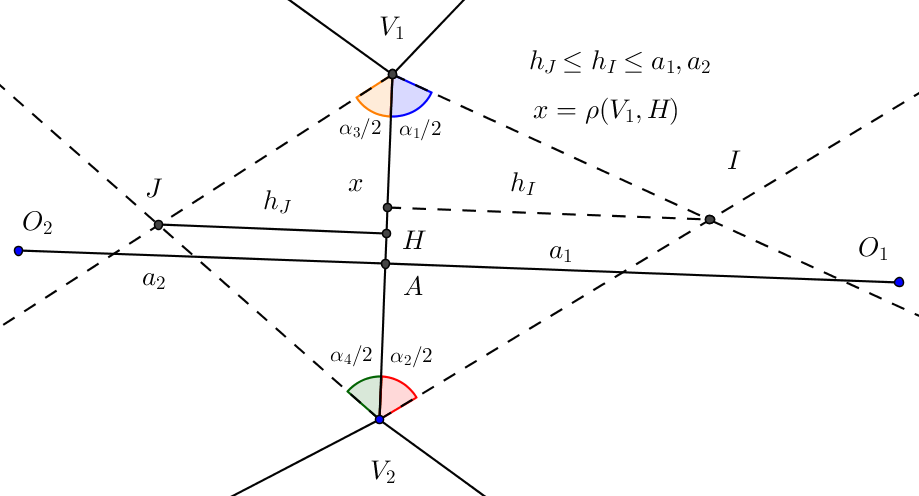}
\caption{The scan}
\label{fig3}
\end{center}
\end{figure}

We have $\Pi(a_1), \Pi(a_2) \leq \Pi(h_J)$. It follows that
$\arctan(e^{-h_J}) > \alpha_{12}/4$. Thus,
$$
h_J < \ln \biggl(\cot \biggl(\frac{\alpha_{12}}{4} \biggr)\biggr).
$$

We introduce the notation $A_0 := \operatorname{tanh} \bigl(\ln(\cot({\alpha_{12}}/{4}))\bigr)$. Then
$\operatorname{tanh} h_J < A_0$. Let $H$ be the projection of the point $J$ onto the edge $E$ and let
$x=\rho(H,V_1)$.

From the right triangles $V_1 J H$ and $V_2 J H$ we find
\begin{equation}
\label{eq4}
\operatorname{tanh} h_J=\tan \biggl(\frac{\alpha_4}{2}\biggr) \sinh(a-x) =
\tan \biggl(\frac{\alpha_3}{2} \biggr) \sinh x,
\end{equation}
which implies that
$$
\sinh x=\frac{\operatorname{tanh} h_J}{\tan ({\alpha_3}/{2})} < \frac{A_0}{\tan ({\alpha_3}/{2})},
\qquad \sinh (a-x) < \frac{A_0}{\tan ({\alpha_4}/{2})}.
$$
Hence,
$$
a=x+(a-x) <\operatorname{arcsinh} \biggl(\frac{A_0}{\tan ({\alpha_3}/{2})}\biggr)
+ \operatorname{arcsinh} \biggl(\frac{A_0}{\tan ({\alpha_4}/{2})}\biggr),
$$
which completes the proof.
 \end{proof}

\subsection{The proof of Theorem~\ref{t1.1} and bounds on  $|(u_3,
u_4)|$}
\label{ss2.4}
Let a polyhedron $P$ be the fundamental polyhedron of a
$\mathbb{Q}$-arithmetic cocompact reflection group in the three-dimensional
Lobachevsky space and let $E$ be the outermost edge of the polyhedron
$P$. Consider the set of unit outer normals  $(u'_1, u'_2, u'_3,
u'_4)$ to the faces $F_1$, $F_2$, $F_3$, $F_4$. Note that this vector system is
linearly independent. Its Gram matrix is
$$
G(u'_1, u'_2, u'_3, u'_4) =
\begin{pmatrix}
1 & -\cos \alpha_{12} & -\cos \alpha_{13} & -\cos \alpha_{14}\\
-\cos \alpha_{12} & 1 & -\cos \alpha_{23} & -\cos \alpha_{24}\\
-\cos \alpha_{13} & -\cos \alpha_{23} & 1 & -T\\
-\cos \alpha_{14} & -\cos \alpha_{24} & -T & 1
\end{pmatrix},
$$
where $T=|(u'_3, u'_4)|=\cosh \rho(F_3, F_4)$ in the case where the faces $F_3$ and
$F_4$ diverge. Recall that otherwise $T \leq 1$, and we do not need to consider this case separately.

Let $(u^*_1, u^*_2, u^*_3, u^*_4)$ be the basis dual to the basis $(u'_1,
u'_2, u'_3, u'_4)$. Then $u^*_3$ and $u^*_4$ determine the vertices $V_2$ and $V_1$ in the Lobachevsky space.
Indeed, the vector $v_1$ corresponding to the point
$V_1 \in \mathbb{L}^3$ is uniquely determined (up to scaling)  by the conditions
$(v_1, u'_1)=(v_1, u'_2)=(v_1, u'_3)=0$. Note that
the vector $u_4^*$ satisfies the same conditions. Therefore, the vectors $v_1$ and
$u^*_4$ are proportional, hence,
$$
\cosh a=\cosh \rho(V_1, V_2)=-(v_1, v_2)=-\frac{(u_3^*,u_4^*)}
{\sqrt{(u_3^*,u_3^*)(u_4^*,u_4^*)}}.
$$

It is known that $G(u^*_1, u^*_2, u^*_3, u^*_4)=G(u'_1, u'_2, u'_3, u'_4)^{-1}$,
whence it follows that $\cosh a$ can be expressed in terms of the algebraic complements
$G_{ij}$ of the elements of the matrix $G=G(u'_1, u'_2, u'_3, u'_4)$:
$$
\cosh a=-\frac{(u_3^*,u_4^*)}{\sqrt{(u_3^*,u_3^*)(u_4^*,u_4^*)}} =
\frac{G_{34}}{\sqrt{G_{33} G_{44}}}.
$$

Denote the right-hand side of the inequality from Theorem~\ref{t2.2} by $F(\overline
\alpha)$, where $\overline \alpha=(\alpha_{12}, \alpha_{13}, \alpha_{23},
\alpha_{14}, \alpha_{24})$; then this theorem implies that $\cosh
a < \cosh F(\overline \alpha)$. It follows that
\begin{equation}
\label{eq5}
\frac{G_{34}}{\sqrt{G_{33} G_{44}}} < \cosh F(\overline \alpha).
\end{equation}
For every $\overline \alpha$, in this way we obtain  a linear inequality
with respect to the number $T$.

\begin{lemma}
\label{l2.1}
The following relations are true:

{\rm(i)}~$\alpha_{12}+\alpha_{23}+\alpha_{13} > \pi,
\quad \alpha_{12}+\alpha_{24}+\alpha_{14} > \pi$;

{\rm(ii)}
\begin{alignat*}{2}
\cos \alpha_1&=\frac{\cos \alpha_{23}+\cos \alpha_{12} \cdot \cos \alpha_{13}}
{\sin \alpha_{12} \cdot \sin \alpha_{13}}, &\qquad
\cos \alpha_2&=\frac{\cos \alpha_{24}+\cos \alpha_{12} \cdot \cos \alpha_{14}}
{\sin \alpha_{12} \cdot \sin \alpha_{14}},
\\
\cos \alpha_3&=\frac{\cos \alpha_{13}+\cos \alpha_{12} \cdot \cos \alpha_{23}}
{\sin \alpha_{12} \cdot \sin \alpha_{23}}, &\qquad
\cos \alpha_4&=\frac{\cos \alpha_{14}+\cos \alpha_{12} \cdot \cos \alpha_{24}}
{\sin \alpha_{12} \cdot \sin \alpha_{24}}.
\end{alignat*}
\end{lemma}

\begin{proof}
To prove both parts of the lemma, we intersect each trihedral angle
with the vertices $V_1$ and $V_2$ by spheres centered at these points.
In the intersection we obtain spherical triangles the angles  of
which are the dihedral angles $\alpha_{ij}$, and the lengths of their edges
are the flat angles $\alpha_k$. This  implies at once the first assertion, and
the second one follows from the dual cosine-theorem for these triangles (see, for
example, \cite[p.\,71]{AVS88}). The lemma is proved.
\end{proof}

\begin{figure}[!htp]
\begin{center}
\includegraphics[scale=1.2]{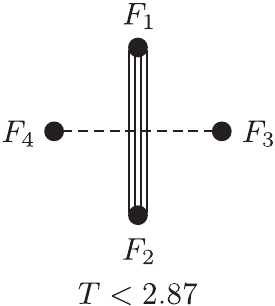}
\caption{Coxeter diagram of the edge with angles
$(\pi/6, \pi/2, \pi/2, \pi/2, \pi/2)$}
\label{fig4}
\end{center}
\end{figure}

It is known that the dihedral angles of the fundamental polyhedron of an arithmetic
hyperbolic reflection group with a ground field $\mathbb{Q}$ can equal only
$\pi/2$, $\pi/3$, $\pi/4$, and $\pi/6$.

It is easy to verify that, taking into account Lemma~\ref{l2.1},\,(i), there are exactly
$44$ different (up to numbering) sets of angles $\overline \alpha$.
For each such set  $\overline \alpha$ inequality~\eqref{eq5} gives some
bound $T < t_{\overline \alpha}$.

\begin{table}
\begin{center}
\caption{Coxeter diagrams of the outermost edge}
\vskip1mm
\includegraphics[scale=1.2]{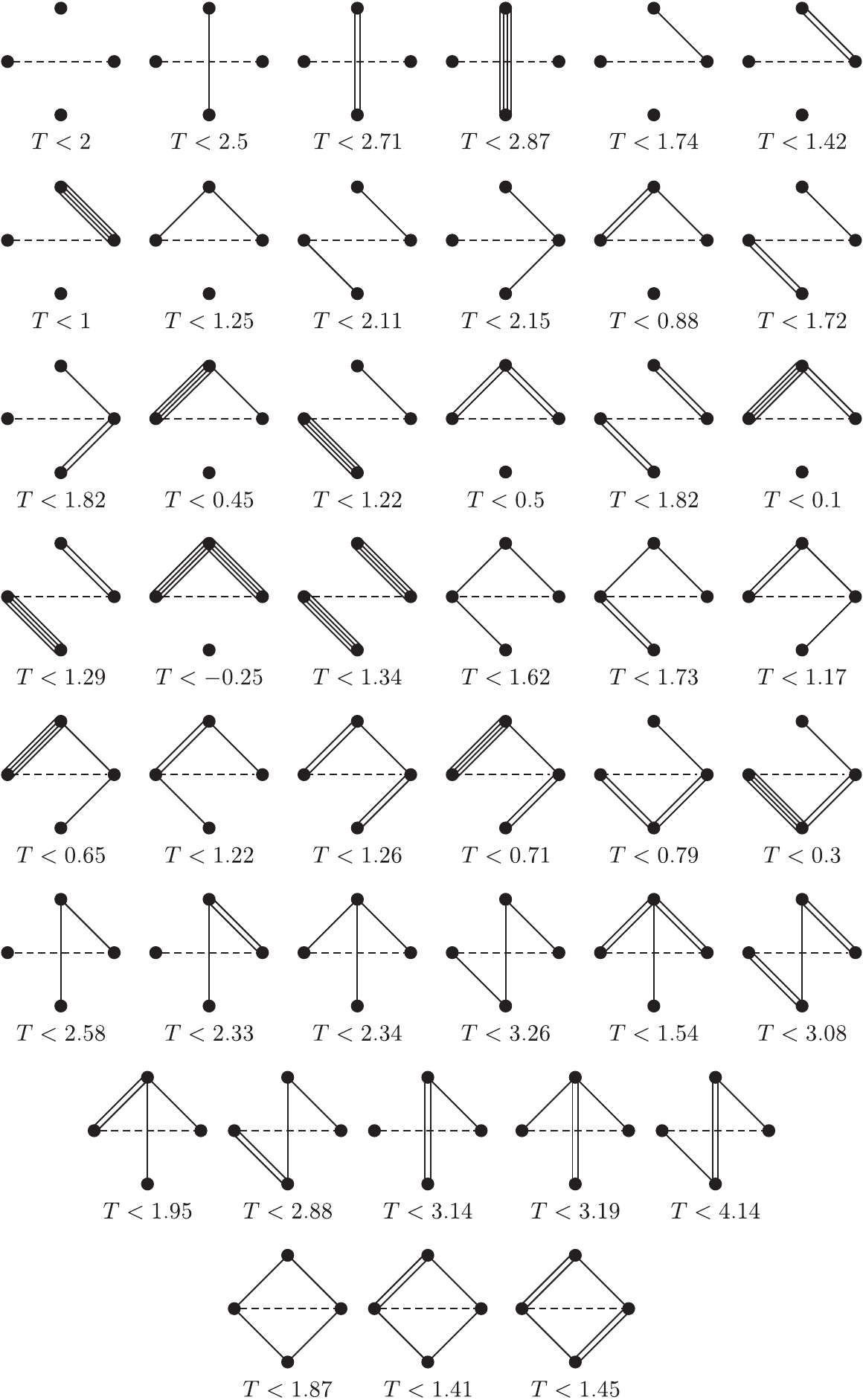}
\label{tab}
\end{center}
\end{table}

 For solving the $44$ linear inequalities a program was compiled
 in the computer algebra system Sage\footnote{
The Sage Developers, \textit{the Sage Mathematics Software System
(Version~7.6)}, SageMath, \href{http://www.sagemath.org}{http:/\!/www.sagemath.org}, 2017},
its code is available on the Internet\footnote{
N.~Bogachev, \textit{Method of the outermost edge/bounds},
\href{https://github.com/nvbogachev/OuterMostEdge/blob/master/bounds.sage/}
{https:/\!/github.com/nvbogachev/\allowbreak OuterMostEdge\allowbreak
/blob/master/bounds.sage/}, 2017}.

The obtained results are presented in Table~\ref{tab} in the form of a set of
Coxeter diagrams for the faces  $F_1$, $F_2$, $F_3$, $F_4$. The faces $F_3$ and $F_4$ will be
connected by a dotted line, and the whole diagram will be signed by the relevant
bound: $T < t_{\overline \alpha}$.

An example of how an edge diagram looks like for~$\overline \alpha=(\pi/6, \pi/2,
\pi/2, \pi/2, \pi/2)$ see in Figure~\ref{fig4}. In this figure we see that
$t_{(\pi/6, \pi/2, \pi/2, \pi/2, \pi/2)}=2.87$.

Thus, we have proved the following theorem.

\begin{theorem}
\label{t2.3}
The fundamental polyhedron of every $\mathbb{Q}$-arithmetic cocompact
 reflection group in~$\mathbb{L}^3$ contains an edge of width less than $t_{\overline
\alpha}$, where $t_{\overline \alpha}$ is the number (depending on the set $\overline
\alpha$), specified in Table~\ref{tab}. Moreover,
$$
\max_{\overline \alpha} \{t_{\overline \alpha}\}=t_{(\pi/4, \pi/2, \pi/3,
\pi/3, \pi/2)}= 4.14.
$$
\end{theorem}

The numbers given in Table~\ref{tab} were calculated on a computer with accuracy
up to eight decimal places. The table shows them rounded up to the nearest hundredth, which is 
quite enough for our purposes. \textit{Note: the numbering of the faces on each diagram is
the same as in the Figure~\ref{fig2}.}

Let $E$ be the edge (chosen in Theorem~\ref{t2.3}) of width less than  $t_{\overline
\alpha}$ for some~$\overline \alpha$ in the fundamental polyhedron $P$ of the group
 $\mathcal{O}_r (L)$, and let $(u_3, u_3)=k$, $(u_4, u_4)=l$. Then
\begin{equation}
\label{eq6}
|(u_3, u_4)| < t_{\overline \alpha} \cdot \sqrt{kl}.
\end{equation}
Note that this bound is much better than estimate~\eqref{eq2}.

\section{Quadratic lattices}
\label{s3}

In this section we give some necessary information about
indefinite quadratic lattices. For more details, see, for example,~\cite{Kas78} or \cite{Vin84}.

Let $A$ be a principal ideal ring. A \textit{quadratic $A$-module} is a
free $A$-module of finite rank equipped with a non-degenerate
 symmetric bilinear form
with values in $A$, called an \textit{inner product}. In particular,
a quadratic $\mathbb{Z}$-module is called
a \textit{quadratic lattice}.
We denote by $[C]$ the standard module $A^n$ whose
inner product is defined by a Gram matrix $C$.
\textit{A quadratic space} is a free $\mathbb{F}$-module of finite rank over $\mathbb{F}$.

The determinant of the Gram matrix of a basis of a module $L$  is called
a \textit{discriminant} $d (L)$ of the quadratic $A$-module $L$.
It is defined up to multiplication by an element of $(A^*)^2 $ ($A^*$
denotes the group of invertible elements of the ring $A$) and can be regarded as
an element of the semigroup $A /(A^*)^2$.
A quadratic $A$-module $L$ is called \textit{unimodular} if $d(L) \in A^*$.

A nonzero vector $x \in L $ is called \textit{isotropic} if $(x, x) = 0$.
A quadratic
module $L$ is called \textit{isotropic} if it contains at least one
isotropic vector,
otherwise $L$ is called \textit{anisotropic}.

Since $(\mathbb{Z}^*)^2 = {1}$, the discriminant $d (L)$ of a
quadratic lattice $L$
is  an integer number. The unimodularity of a quadratic lattice $L$ is
equivalent to the property that $L$ coincides with its \textit{conjugate} lattice
$$
L^* = \{x \in L \otimes \mathbb{Q} \mid (x, y)
\in \mathbb {Z} \ \  \forall y \in L \}.
$$

For a lattice $L$, the invariant factors of the Gram matrix of a basis of $L$
are called \textit {invariant factors} of the lattice $L$.
Every quadratic lattice $L$ generates the quadratic real vector
space
$L_{\infty} = L \otimes \mathbb{R}$ and, for any prime~$p$, the quadratic $\mathbb{O}_p$-module
$L_p = L\,{\otimes}\,\mathbb{O}_p$, where $\mathbb {O}_p$ is the ring
of $p$-adic numbers.
The signature of
 the lattice $ L $ is defined as the signature of the space $ L_{\infty}$.
 It is obvious that if two quadratic
lattices $L$ and $M$ are isomorphic, then
they have the same signature and $ L_p \simeq M_p $ for any prime $p$.
The converse is also true under the following conditions:

(i) $L$ is indefinite;

(ii) for any prime $p$, the lattice $L$ has two invariant factors
divisible by  the same power of $p$.

The structure of quadratic $\mathbb{O}_p$-modules
can be described as follows.
Each such module
$L_p$ admits the \textit{Jordan decomposition}
$$
L_p = L_p^{(0)} \oplus [p] L_p^{(1)} \oplus [p^2] L_p^{(2)} \oplus \cdots,
$$
where all $L_p^{(j)}$ are unimodular quadratic $\mathbb{O}_p$-modules.
These unimodular modules are determined by $L$ uniquely up to an isomorphism, unless
$p \not = 2$. In the case $p = 2$ the rank and the parity of each such module are uniquely
determined by $L$.

\begin{proposition}[{\rm(See, for example, \cite{Vin84})}]
\label{p3.1}
If $L$ is a maximal quadratic lattice that is not contained
in any other quadratic lattice, then
$$
L_p=L_p^{(0)} \oplus [p] L_p^{(1)}
$$
for all primes $p\mid d(L)$.
\end{proposition}

\begin{definition}
\label{d3.1}
Let $a, b \in \mathbb{Q}^*_p$. Set
$$
(a, b)_p =
\begin{cases}
1 & \text{if the equation } ax^2 + by^2 = 1 \text{ has a solution in }\mathbb{Q}^*_p,
\\
-1 & \text{otherwise.}
\end{cases}
$$
The number $(a, b)_p $ is called the \textit{Hilbert symbol}.
\end{definition}

It is known that the group $ \mathbb{Q}^*_ p / (\mathbb{Q}^*_ p)^2$ can be
regarded as a vector space over $\mathbb{Z}_2 = \mathbb{Z} / 2 \mathbb{Z}$
of rank $2$  for $p \not = 2$
(respectively, of rank $3$ for $p = 2$).
The Hilbert symbol is a non-degenerate symmetric bilinear
form on this vector space.

Recall the definition of a Hasse invariant for an arbitrary quadratic
space $V$ over the field $\mathbb{Q}_p$. Let $a_1, \dots, a_m$ be the squares of lengths
of vectors of some orthogonal basis of the space $V$.

\begin{definition}
The number
$$
\varepsilon_p (f) = \prod_{i <j} (a_i, a_j)_p
$$
is called the Hasse invariant of the quadratic space $V$.
\end{definition}

The following assertions are known.

\begin{theorem}[{\rm See, for example, \cite[Lemma~2.6]{Kas78}}]
\label{t3.1}
A quadratic space $V$ of rank $4$ over the field $\mathbb{Q}_p$
is anisotropic if and only if the following conditions hold:

{\rm(1)}~$d(V) \in (\mathbb{Q}^*_p)^2$;

{\rm(2)}~$\varepsilon_p (V)=-(-1, -1)_p$.
\end{theorem}

\begin{theorem}[{\rm Strong Hasse principle, see, for example,
\cite[Theorem~1.1]{Kas78}}]
\label{t3.2}
A quadratic space $V$ over the field $\mathbb{Q}$ is isotropic if and only if
$V \otimes \mathbb{Q}_p$ is isotropic for all $p$, including $\infty$.
\end{theorem}

\begin{theorem}[{\rm Weak Hasse principle, see, for example,
\cite[Theorem~1.2]{Kas78}}]
\label{t3.3}
Two rational quadratic spaces are isomorphic over $\mathbb{Q}$ if and only if
they are  isomorphic over $\mathbb{R}$ and over all $\mathbb{Q}_p$.
\end{theorem}

\section{Methods of testing a lattice for $(1{,}2)$-reflectivity}
\label{s4}

\subsection{Vinberg's algorithm}
\label{ss4.1}

As it was said before, in 1972, Vinberg suggested an algorithm of constructing the fundamental
polyhedron of a hyperbolic reflection group. This algorithm is theoretically applicable to
any hyperbolic reflection group, but practically it is efficient only for groups of the
form $\mathcal{O}_r (L)$ (and also of the form $\mathcal{O}^{(2)}_r (L)$, etc.).

In this subsection we describe Vinberg's algorithm, following \cite{Vin72} and \cite{Vin73}.
We pick a point $v_0 \in \mathbb{L}^n$, which we shall call a
\textit{a basic point}. The fundamental domain $P_0$ of its stabilizer
$\mathcal{O}_r(L)_{v_0}$ is a polyhedral cone in $\mathbb{L}^n$. Let
$H_1, \dots , H_m$ be the sides of this cone and let $a_1, \dots , a_m$ be the
corresponding outer normals. We define the half-spaces
$$
H_k^-=\{x \in \mathbb{E}^{n,1}\mid (x, a_k) \le 0 \}.
$$
Then $P_0=\bigcap_{j=1}^m H_j^-$.

There is the unique fundamental polyhedron $P$ of the group
$\mathcal{O}_r(L)$ contained in~$P_0$ and
containing the point $v_0$. Its faces containing $v_0$ are formed by the cone faces
$H_1, \ldots, H_m$. The other faces $H_{m+1}, \ldots $ and the corresponding outer normals
 $a_{m+1} , \ldots $ are constructed by induction. Namely, for $H_j$ we take a mirror
 such that the root $a_j$ orthogonal to it satisfies the conditions:

1)~$(v_0, a_j) < 0$;

2)~$(a_i, a_j) \le 0$ for all $i < j$;

3)~the distance $\rho(v_0, H_j)$  is minimal subject to constraints~1) and~2).

The lengths of the roots of a lattice $L$ satisfy the following condition.

\begin{proposition}[{\rm(Vinberg, see \cite[Proposition~24]{Vin84})}]
\label{p4.1}
The squares of lengths
of roots in the quadratic lattice $L$ are divisors
of the doubled last invariant factor of $L$.
\end{proposition}

The fundamental polyhedron of a group $\mathcal{O}_r (L)$ is a Coxeter polyhedron and is
determined by its Coxeter diagram. A polyhedron is of finite volume if and only if it is
the convex span of a finite number of usual points or points at infinity of the space $\mathbb{L}^n$.
Every vertex of a finite-volume Coxeter polyhedron corresponds either to an elliptic subdiagram
or rank $n$ (usual vertices) of its Coxeter diagram or to a parabolic subdiagram
or rank $n-1$ (vertices at infinity). Thus, according to the Coxeter diagram, one can determine
whether a polyhedron has a finite volume. For more details, see, for example, Vinberg's papers
 \cite{Vin67} and \cite{Vin85}.

Some efforts to implement Vinberg's algorithm by using a computer have been made since
the 1980s, but they all dealt with particular lattices, usually with an orthogonal basis.
Such programs are mentioned, e.g., in the papers of Bugaenko (1992, see \cite{Bug92}),
Scharlau and Walhorn (1992, see \cite{SW92}),
Nikulin (2000, see \cite{Nik00}), and Allcock (2011, see \cite{All12}). But the programs
themselves have not been published; the only
exception is Nikulin's paper, which contains a program code for lattices of several different special
forms. The only known implementation published along with a detailed documentation is
Guglielmetti's 2016
program\footnote{see\ \href{https://rgugliel.github.io/AlVin}{https:/\!/rgugliel.github.io/AlVin}},
processing hyperbolic lattices with an orthogonal basis with square-free invariant
factors over several ground fields.
Guglielmetti used this program in his thesis (2017, see \cite{Gug17}) to classify reflective hyperbolic lattices
with an orthogonal basis with small lengths of its elements.
His program works fairly well in all dimensions in which
reflective lattices exist.

In this paper, we use the program created in 2017 by the author jointly with
A.\,Yu.~Perepechko. This program is available on the Internet
(see \cite{VinAlg2017}), and one can find its detailed description in \cite{BP18}.

\subsection{The method of ``bad'' reflections}
\label{ss4.2}

If we can construct the fundamental polyhedron of the group $O_r (L)$
for some reflective lattice $L$, then it is easy to determine whether it is $(1{,}2)$-reflective.
One can consider the group $\Delta$ generated by the $k$-reflections for $k > 2$ (we shall call them
``bad'' reflections)
in the sides of the fundamental polyhedron of the group $O_r (L)$.
The following lemma holds (see~\cite{Vin07}).

\begin{lemma}
\label{l4.1}
A lattice $L$ is $(1{,}2)$-reflective if and only if it is reflective and the group
$\Delta$ is finite.
\end{lemma}

Actually, to prove that a lattice is not $(1{,}2)$-reflective,
it is sufficient to construct only some part of the fundamental polyhedron containing 
an infinite subgroup generated by bad reflections.

\section{Short list of candidate-lattices}
\label{s5}

\subsection{Plan for finding a short list}
\label{ss5.1}
Let $P$ be the fundamental polyhedron of the group $\mathcal{O}^{(1{,}2)}_r (L)$
for an anisotropic hyperbolic lattice $L$ of rank $4$.
This lattice is $(1{,}2)$-reflective if and only if $P$ is compact.
By Theorem~\ref{t2.3},
every such polyhedron contains an edge $E$ of width less than
 $t_{\overline \alpha}$, where $t_{\overline \alpha} \leq 4.14$ is the number
 depending on the set $\overline \alpha$ of dihedral angles around this edge.

Let $u_1$, $u_2$, $u_3$, $u_4$ be the roots of the lattice $L$ that are the outer
normals to the faces $F_1$, $F_2$, $F_3$, $F_4$, respectively. These roots generate
some sublattice
$$
L'=[G(u_1, u_2, u_3, u_4)] \subset L.
$$

Note that the elements of the Gram matrix $G(u_1, u_2, u_3, u_4)$
can assume only finitely many different values.
Namely, the diagonal elements can equal only~$1$ or $2$,
and the absolute values of the remaining elements $g_{ij}$ must be
strictly less than $\sqrt{g_{ii} g_{jj}}$, excepting $g_{34}=(u_3,
u_4)$, whose absolute value is bounded by $t_{\overline \alpha} \sqrt{(u_3,u_3) (u_4,
u_4)}$.

Thus, we obtain a finite list of matrices $G(u_1, u_2, u_3, u_4)$.
We pick in these matrices the ones that define anisotropic lattices, and
after that we find all their possible extensions.

In order to select only anisotropic lattices, we use a computer program%
\footnote{N.~Bogachev, \textit{Method of the outermost edge/is\_anisotropic},
\href{https://github.com/nvbogachev/OuterMostEdge/blob/master/is\_anisotropic}
{https:/\!/github.com/\allowbreak nvbogachev/\allowbreak OuterMostEdge/\allowbreak
blob/\allowbreak master/is\_anisotropic}, 2017},
based on the results and methods formulated in~\S\,\ref{s3} of this paper.

We split the list of all anisotropic lattices into isomorphism classes and
take only one representative of each class.  So, now we obtain
a substantially shorter list of anisotropic lattices that are pairwise non-isomorphic.

After that we find all their finite-index extensions
and verify the resulting list of candidate-lattices on $(1{,}2)$-reflectivity
using the methods described in~\S\,\ref{s4}.

\subsection{Short list of candidate-lattices}

The final program that  creates a list of numbers $t_{\overline \alpha}$ and
then, using
this list, displays all Gram matrices $G(u_1, u_2, u_3, u_4)$, is also
available on the Internet%
\footnote{N.~Bogachev, \textit{Method of the outermost edge/CandidatesFor12Reflectivity},
\href{https://github.com/nvbogachev/OuterMostEdge/blob/master/Is_equival}
{https:/\!/github.com/\allowbreak nvbogachev/\allowbreak
OuterMostEdge/blob/master/Is\_equival}, 2017}.

As the output we obtain matrices  $G_1$--$G_{7}$, for each of which we find all corresponding 
extensions. 

To each Gram matrix $G_k$ in our notation, there corresponds a lattice $L_k$ that can have some
other extensions. For each new anisotropic lattice 
(non-isomorphic to any previously found lattice) we introduce the notation 
$L(k)$, where $k$ denotes its number:
{\allowdisplaybreaks 
\begin{align*}
G_{1} &=
\begin{pmatrix}
1 & 0 & 0 & -1\\
0 & 2 & 0 & -1\\
0 & 0 & 1 & -2\\
-1 & -1 & -2 & 2
\end{pmatrix}, \qquad L_1 \simeq [-7] \oplus [1] \oplus [1] \oplus [1] := L(1),
\\
G_{2} &=
\begin{pmatrix}
1 & -1 & 0 & 0\\
-1 & 2 & 0 & -1\\
0 & 0 & 1 & -4\\
0 & -1 & -4 & 2
\end{pmatrix}, \qquad L_2 \simeq [-15] \oplus [1] \oplus [1] \oplus [1] := L(2),
\\
G_{3} &=
\begin{pmatrix}
1 & 0 & 0 & -1\\
0 & 2 & -1 & 0\\
0 & -1 & 2 & -3\\
-1 & 0 & -3 & 2
\end{pmatrix}, \qquad L_3 \simeq [-3] \oplus [5] \oplus [1] \oplus [1] := L(3),
\\
G_{4} &=
\begin{pmatrix}
2 & 0 & 0 & -1\\
0 & 2 & -1 & -1\\
0 & -1 & 1 & -2\\
-1 & -1 & -2 & 2
\end{pmatrix}, \qquad L_4 \simeq [-23] \oplus [1] \oplus [1] \oplus [1] := L(4),
\\
G_{5} &=
\begin{pmatrix}
2 & -1 & 0 & -1\\
-1 & 2 & -1 & 0\\
0 & -1 & 1 & -4\\
-1 & 0 & -4 & 2
\end{pmatrix}, \qquad L_5 \simeq [-55] \oplus [1] \oplus [1] \oplus [1] := L(5),
\\
G_{6} &=
\begin{pmatrix}
2 & 0 & 0 & -1\\
0 & 2 & 0 & -1\\
0 & 0 & 2 & -3\\
-1 & -1 & -3 & 2
\end{pmatrix}, \qquad L_6=[G_{6}] := L(6),
\\
G_{7} &=
\begin{pmatrix}
2 & 0 & -1 & -1\\
0 & 2 & -1 & -1\\
-1 & -1 & 2 &-3\\
-1 & -1 & -3 & 2
\end{pmatrix} , \qquad L_7=[G_{7}] := L(7).
\end{align*}
} 

The lattices  $L(1)$--$L(5)$ are maximal. The lattice $L(6)$ has a unique extension of index  $2$
generated by the vectors
\begin{equation}
\label{eq7}
\biggl\{\frac{e_1+e_2}{2}, \frac{e_1-e_2}{2}, e_3, e_4\biggr\},
\end{equation}
where $\{e_1, e_2, e_3, e_4\}$ is a basis of the lattice $L(6)$. The Gram matrix of system 
\eqref{eq7} is equivalent to the matrix $\mathrm{diag}(-7,1,1,1)$, whence it follows that
the only extension in this case is the lattice $L(1)$. Similarly, for the lattice 
 $L(7)$ we find a unique extension of index $2$ generated by the vectors
\begin{equation}
\label{eq8}
\biggl\{\frac{e_1+e_2}{2}, \frac{e_1-e_2}{2}, e_3, e_4\biggr\}
\end{equation}
that is isomorphic to $L(2)$; here $\{e_1, e_2, e_3, e_4\}$ is a basis of the lattice  $L(7)$.

We observe that the lattices $L(6)$ and $L(7)$ are (unique) even sublattices of index $2$ 
of the lattices $L(1)$ and $L(2)$, respectively. 

\section{Verification of $(1{,}2)$-reflectivity and proof of Theorem~\ref{t1.2}}
\label{s6}

It remains to verify for $(1{,}2)$-reflectivity a small number of lattices. 
The lattices  $L(1)$, $L(2)$, $L(6)$ and $L(7)$ are
$2$-reflective (see \cite{Vin07}), hence, are $(1{,}2)$-reflective.

\begin{proposition}
\label{p6.1}
The lattice  $L(3)=[-3] \oplus [5] \oplus [1] \oplus [1]$ is reflective, but not
$(1{,}2)$-reflective.
\end{proposition}

\begin{proof}
For the lattice  $L(3)$ we apply Vinberg's algorithm. Our program finds seven roots:
\begin{gather*}
\begin{alignedat}{4}
a_1&=(0;0,0,-1), &\quad (a_1,a_1)&=1, &\qquad a_2&=(0;0,-1,1), &\quad (a_2,a_2)&=2,
\\
a_3&=(0;-1,0,0), &\quad (a_3,a_3)&=5, &\qquad a_4&=(1;0,3,0), &\quad (a_4,a_4)&=6,
\\
a_5&=(1;1,0,0), &\quad (a_5,a_5)&=2, &\qquad a_6&=(2;1,2,2), &\quad (a_6,a_6)&=1,
\end{alignedat}
\\
a_7=(10;6,10,5), \quad (a_7,a_7)=5.
\end{gather*}
The Gram matrix of these roots has the form
$$
G(a_1,a_2,a_3,a_4,a_5,a_6,a_7) =
\begin{pmatrix}
1 & -1 & 0 & 0 & 0 & -2 & -5\\
-1 & 2 & 0 & -3 & 0 & 0 & -5\\
0 & 0 & 5 & 0 & -5 & -5 &-30\\
0 & -3 & 0 & 6 & -3& 0 & 0\\
0 & 0 &-5 & -3 & 2 & -1 & 0\\
-2 & 0 & -5 & 0 & -1 & 1 & 0\\
-5 & -5 & -30 & 0 & 0 & 0 & 5
\end{pmatrix}.
$$
The Coxeter diagram corresponding to this Gram matrix is represented in Figure~\ref{fig5}.

\begin{figure}[!htp]
\begin{center}
\includegraphics{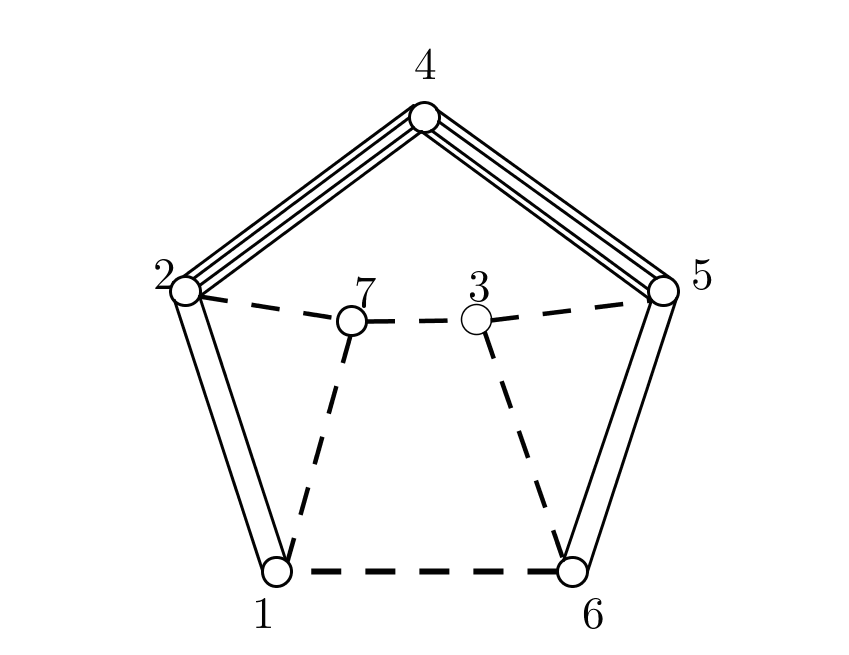}
\caption{The Coxeter diagram of the fundamental polyhedron of the lattice $L(3)$}
\label{fig5}
\end{center}
\end{figure}

It determines a three-dimensional Coxeter polyhedron of finite volume. This diagram has no
parabolic subdiagrams, so this polyhedron is bounded 
(however, this is verified by CoxIter, created and published by 
R.~Guglielmetti). We observe that the roots  $a_3$, $a_4$, $a_7$ determine a group generated 
by ``bad'' reflections, it is infinite, since the corresponding subdiagram 
(its vertices are colored in black) contains a dotted edge. Therefore, the lattice 
$L(3)$ is reflective, but not $(1{,}2)$-reflective. The proposition is proved. 
\end{proof}

The non-reflectivity of the lattice $L(4)$ was proved in~\cite{Mark15}
(see also the dissertation \cite{Mcl13}).

\begin{proposition}
\label{p6.2}
The lattice  $L(5)=[-55] \oplus [1] \oplus [1] \oplus [1]$ is not 
$(1{,}2)$-reflective. 
\end{proposition}

\begin{proof}
It turns out that for verification of the $(1{,}2)$-reflectivity there is no need to complete Vinberg's 
algorithm. The program we use finds the first eight roots 
\begin{alignat*}{4}
a_1&=(0, -1, -1, 0), &\quad (a_1,a_1)&=2;&\qquad a_2&=(0, 0, 1, -1), &\quad (a_2,a_2)&=2;
\\
a_3&=(0, 1, 0, 0), &\quad (a_3,a_3)&=1;&\qquad a_4&=(2, 0, 11, 11), &\quad (a_4,a_4)&=22;
\\
a_5&=(1, -4, 4, 5), &\quad (a_5,a_5)&=2;&\qquad a_6&=(1, -2, 2, 7), &\quad (a_6,a_6)&=2;
\\
a_7&=(2, -5, 10, 10), &\quad (a_7,a_7)&=5;&\qquad a_8&=(2, 0, 0, 15), &\quad
(a_8,a_8)&=5;
\end{alignat*}
the Gram matrix of which has the form
$$
G(a_1,a_2,a_3,a_4,a_5,a_6,a_7,a_8){=}\!\!\begin{pmatrix}
2 & -1 & -1 & -11 & 0 & 0 &-5 & 0\\
 -1 & 2 & 0 & 0 & -1 & -5 & 0 &-15\\
-1 & 0 & 1 & 0 & -4 &-2 &-5 & 0 \\
-11 & 0 & 0 & 22 & -11 & -11 & 0 & -55\\
0 & -1& -4 & -11 & 2 & -4 & 0 & -35 \\
 0 & -5 & -2& -11& -4 & 2 & -10 & -5\\
-5 & 0& -5 & 0 & 0 &-10 & 5 &-70\\
0 & -15 & 0 & -55 & -35 & -5 & -70 & 5
\end{pmatrix}.
$$

It suffices to consider the subgroup generated by ``bad'' reflections with respect to the mirrors 
 $H_{a_7}$ and $H_{a_8}$. Since these mirrors diverge, this subgroup is infinite. The proposition is 
 proved. 
\end{proof}

Thus, among the seven anisotropic candidate-lattices selected in the process of solution only four 
are $(1{,}2)$-reflective. These are the lattices $L(1)$, $L(2)$, $L(6)$, $L(7)$. 
Theorem \ref{t1.2} is proved. 


\end{document}